\documentclass[12pt]{article}
\usepackage{array}
\usepackage{amsthm,amsmath,amssymb,color}
\usepackage{fullpage}

\usepackage{color}
\usepackage[normalem]{ulem}

\newtheorem{thm}{Theorem}[section]
\newtheorem{theorem}[thm]{Theorem}
\newtheorem{lemma}[thm]{Lemma}
\newtheorem{obs}[thm]{Observation}

\newtheorem{cons}[thm]{Construction}
\newtheorem{ex}[thm]{Example}

\begin{document}

\title{Sharp bounds for the Randi{\' c} index of graphs with given minimum and maximum degree}
\author{Suil O\thanks{Department of Applied Mathematics and Statistics, The State University of New York, Korea, Incheon, 21985, suil.o@sunykorea.ac.kr}\, and
Yongtang Shi\thanks{Center for Combinatorics and LPMC, Nankai University, Tianjin, 300071, China, shi@nankai.edu.cn}
}

\maketitle

\begin{abstract}
The Randi{\' c} index of a graph $G$, written $R(G)$, is the sum of $\frac 1{\sqrt{d(u)d(v)}}$ over all edges $uv$ in $E(G)$.
Let $d$ and $D$ be positive integers $d < D$.
In this paper, we prove that
if $G$ is a graph with minimum degree $d$ and maximum degree $D$, then $R(G) \ge \frac{\sqrt{dD}}{d+D}n$; equality holds only when $G$ is an $n$-vertex $(d,D)$-biregular. Furthermore, we show that if $G$ is an $n$-vertex connected graph with minimum degree $d$ and maximum degree $D$, then
$R(G) \le \frac n2- \sum_{i=d}^{D-1}\frac 12 \left( \frac 1{\sqrt{i}} - \frac 1{\sqrt{i+1}}\right)^2$; it is sharp for infinitely many $n$, and we characterize when equality holds in the bound.

\end{abstract}

\section{Introduction}
The \textit{Randi{\'c} index} of a graph $G$, written $R(G)$, is defined as follows:
$$R(G)=\sum_{uv \in E(G)} \frac 1{\sqrt{d(u)d(v)}},$$
where for a vertex $v \in V(G)$, $d(v)$ is the degree of $v$. The concept was introduced by Milan Randi\'c
under the name ``\textit{branching index}" or ``\textit{connectivity index}"  in 1975 \cite{R}, which has a good
correlation with several physicochemical properties of alkanes.  In 1998 Bollob\'as and Erd\"{o}s \cite{BE}
generalized this index by replacing $-\frac{1}{2}$ with any real
number $\alpha$, which is called the general Randi\'c index. There are also many other variants of 
Randi\'c index \cite{DvorakLidickySkrekovski_y2011_p434,KnorLuzarSkrekovski_y2015_p160,Shi}. For more results on 
Randi\'c index,
see the survey paper \cite{LS2007}.

Many important mathematical properties of 
Randi\'c index have been established. Especially, the relations between 
Randi\'c index and other graph parameters have been widely studied, such as the minimum degree \cite{BE}, the chromatic index \cite{LS}, the diameter \cite{DvorakLidickySkrekovski_y2011_p434,YL}, the radius \cite{CH}, the average distance \cite{CH}, the eigenvalues \cite{AD,AHZ}, and the matching number \cite{AHZ}.

In 1988, Shearer proved if $G$ has no isolated vertices then $R(G)\ge \sqrt{|V(G)|}/2$ (see \cite{F}).  A few months
later Alon improved this bound to $\sqrt{|V(G)|}-8$ (see \cite{F}). In 1998, Bollob\'as and
Erd\"{o}s \cite{BE} proved that the Randi\'c index of an $n$-vertex graph $G$ without isolated vertices is at least $\sqrt{n-1}$, with equality if and only if $G$ is a star. In \cite{F},
Fajtlowicz mentioned that Bollob\'as and
Erd\"{o}s asked the minimum value for the Randi\'c index in a graph with given minimum degree. Then the question was answered in various ways \cite{AH,DFR,LLT,LS2008}.

For a graph $G$, we denote its complement by $\overline{G}$. We also denote by $K_n$ the complete graph with $n$ vertices and by $K_{n}-e$ the graph obtained from the complete graph $K_n$ by deleting an edge. A graph is $(a,b)$-biregular if it is bipartite with the vertices of one part all having degree $a$ and the others all having degree $b$.

Aouchiche et al.~\cite{AHZ1} studied the relations between Randi\'c index and the minimum degree, the maximum degree, and the average degree, respectively. They proved that for any connected graph $G$ on $n$ vertices with minimum degree $d$ and maximum degree $D$, then $R(G) \ge \frac{d}{d+D}n$.

In this paper, we prove that
if $G$ is an $n$-vertex graph with minimum degree $d$ and maximum degree $D$, then $R(G) \ge \frac{\sqrt{dD}}{d+D}n$, which improves the result of Aouchiche et al. in \cite{AHZ1}; equality holds only when $G$ is an $n$-vertex $(d,D)$-biregular. Furthermore, we show that if $G$ is an $n$-vertex connected graph with minimum degree $d$ and maximum degree $D$, then
$R(G) \le \frac n2- \sum_{i=d}^{D-1}\frac 12 \left( \frac 1{\sqrt{i}} - \frac 1{\sqrt{i+1}}\right)^2$; it is sharp for infinitely many $n$.

\section{Main Results}

In this section, we first give a sharp lower bound for $R(G)$ in an $n$-vertex graph with givien minimum and maximum degree, improving the one that Aouchiche et al.~\cite{AHZ1} proved.

\begin{theorem}
\label{thm:lower}
 If $G$ is an $n$-vertex graph with minimum degree $d$ and maximum degree $D$, then $R(G) \ge \frac{\sqrt{dD}}{d+D}n$. Equality holds only when $G$ is an $n$-vertex $(d,D)$-biregular.
\end{theorem}
\begin{proof}
For each $i \in \{d,\ldots,D\}$, let $V_i$ be the set of vertices with degree $i$, and let $n_i=|V_i|$. Note that
\begin{equation}\label{sum}
\sum_{i=d}^Dn_i=n.
\end{equation}

Let $m_{ij} = |[V_i,V_j]|$ for all $i, j \in \{d,\ldots,D\}$,
where $[A,B]$ is the set of edges with one end-vertex in $A$ and the other in $B$. Since $G$ has minimum degree $d$ and maximum degree $D$, we have
\begin{equation}\label{randicsum}
R(G)=\displaystyle \sum_{d \le i \le j \le D}\frac{m_{ij}}{\sqrt{ij}}.
\end{equation}
For fixed $i$, the degree sum over all vertices in $V_i$ can be computed by counting the edges between $V_i$ and $V_j$ over all $j\in\{d,\ldots,D\}$;
\begin{equation}\label{partialsum}
in_i = m_{ii}+\sum_{j=d}^{D}m_{ij}.
\end{equation}
Note that $m_{ii}$ must be counted twice.

By manipulating equation~(\ref{partialsum}), we have the followings:

\begin{equation}\label{smalld}
dn_d=(m_{dd}+\sum_{j=1}^{D}m_{dj}) ~~\Rightarrow ~~n_d-\frac{m_{dD}}{d}=\frac 1d(m_{dd}+\sum_{j=d}^{D-1}m_{dj})
\end{equation}

\begin{equation}\label{largeD}
Dn_D=(m_{DD}+\sum_{j=1}^{D}m_{Dj}) ~~\Rightarrow ~~n_D-\frac{m_{dD}}{D}=\frac 1D(m_{DD}+\sum_{j=d+1}^{D}m_{jD})
\end{equation}

\begin{equation}\label{fracsum}
n_i=\frac{1}i(m_{ii}+\sum_{j=d}^{D}m_{ij})
\end{equation}

By equations~(\ref{sum}) and~(\ref{fracsum}), we have
\begin{equation}\label{dD}
n_d+n_D=n-\sum_{i=d+1}^{D-1}n_i=n-\sum_{i=d+1}^{D-1}\frac{1}i(m_{ii}+\sum_{j=d}^{D}m_{ij}).
\end{equation}
By combining equations~(\ref{smalld}),~(\ref{largeD}), and~(\ref{dD}),  we have
$$n_d-\frac{m_{dD}}{d} + n_D-\frac{m_{dD}}{D}=n-\sum_{i=d+1}^{D-1}\frac{1}i(m_{ii}+\sum_{j=d}^{D}m_{ij}) -\left(\frac{d+D}{dD}\right)m_{dD}$$
$$=\frac 1d(m_{dd}+\sum_{j=d}^{D-1}m_{dj})+\frac 1D(m_{DD}+\sum_{j=d+1}^{D}m_{jD}) ~~\Rightarrow$$
$$ \left(\frac{d+D}{dD}\right)m_{dD} = n- \sum_{i=d+1}^{D-1}\frac{1}i(m_{ii}+\sum_{j=d}^{D}m_{ij}) - \frac 1d(m_{dd}+\sum_{j=d}^{D-1}m_{dj})-\frac 1D(m_{DD}+\sum_{j=d+1}^{D}m_{jD})$$
 \begin{equation}\label{mdD}
 \Rightarrow ~~m_{dD}= \frac{dD}{d+D}n-\frac{dD}{d+D}\left[-(\frac 1d+\frac 1D)m_{dD}+\sum_{d \le i \le j \le D} \left(\frac 1i + \frac 1j \right) m_{ij} \right]
\end{equation}

By plugging equation~(\ref{mdD}) into~(\ref{randicsum}), we have
\begin{equation}\label{final}
\sum_{d\le i \le j \le D} \frac{m_{ij}}{\sqrt{ij}}=\frac{\sqrt{dD}}{d+D}n + \sum_{d \le i \le j \le D} \left[\frac 1{\sqrt{ij}} - \frac{\sqrt{dD}}{d+D}\left(\frac 1i + \frac 1j\right)\right]m_{ij}.
\end{equation}
Note that except when $i=d$ and $j=D$, we have $\frac 1{\sqrt{ij}} - \frac{\sqrt{dD}}{d+D}\left(\frac 1i + \frac 1j\right)>0$.
Since $m_{ij}$ is non-negative,
we have $$R(G) \ge \frac{\sqrt{dD}}{d+D}n.$$

If there are vertices $u$ and $v$ such that  $d(u)\neq d$ or $d(v)\neq D$,
then $m_{d(u)d(v)}>0$. Thus the
equality holds only when $G$ is $(d,D)$-biregular.
\end{proof}

From now, we first construct the class of graphs with mimimum degree $d$ and maximum degree $D$ that we will show are those achieving equality in Theorem~\ref{thm:upper}.

\begin{cons}{\rm
		Let $d$ and $D$ be positive integers with $d < D$, and let $H$ be a graph with minimum degree $d$ and maximum degree $D$. Suppose that for $i \in [d,D]$, $V_i(H)$ is the set of vertices with degree $i$ in $V(H)$.
Let ${\cal F}$ be the family of graphs $H$ such that
for $i \in [d,D-1]$, there exists only one vertex in $V_i(H)$ having exactly one neighbor in $V_{i+1}(H)$.
}
\end{cons}

In Example~\ref{ex}, we show that this family is nonempty.

\begin{ex}\label{ex}{\rm
Let $d$ and $D$ be odd positive integers $1\le d < D$. Suppose that
$$V_i=\begin{cases}
K_1 \text{ if } d=1 \text{ and } i=1\\[2mm]
\overline{P_3+\frac{d-1}{2}K_2} \text{ if } d\ge 3 \text{ and } i=d \text{ or } D\\[2mm]
K_{i+1} -e\text{ if } i \in [d+1,D-1].
\end{cases}$$
Note that for $i \in [d,D]$, each vertex in $V_i$ has degree $i$, except for one vertex when $i=d$ or $D$, or two vertices when $i\in [d+1,D-1]$.
For $d\le i \le D-1$, add an edge joining $V_i$ and $V_{i+1}$ so that for $j \in [d,D]$, every vertex in $V_j$ in the resulting graph $F_{d,D}$ has degree $j$.
}
\end{ex}

Recall that Caporossi et al. \cite{CGHP} gave another description of the Randi\'c index by using linear programming.
\begin{theorem}\label{randicexpression}
If $G$ is an $n$-vertex graph without isolated vertices,
then $$R(G) = \frac n2 - \sum_{uv \in E(G)} \frac 12\left(\frac 1{\sqrt{d(u)}} -\frac 1{\sqrt{d(v)}} \right)^2.$$
\end{theorem}

Lemma~\ref{basic} shows that the graph $F_{d,D}$ is included in the family ${\cal F}$.

\begin{lemma} \label{basic}
	If the graph $F_{d,D}$ in Example~\ref{ex} has $n$ vertices, then
$$R(F_{d,D})=\frac n2- \sum_{i=d}^{D-1}\frac 12 \left( \frac 1{\sqrt{i}} - \frac 1{\sqrt{i+1}}\right)^2.$$
\end{lemma}
\begin{proof}
Note that there are exactly $D-d$ edges $uv$ such that $d(u)$ and $d(v)$ are different. In fact, for such an edge $uv$, we have $d(v)=d(u)+1$ if $d(v)>d(u)$. By Theorem~\ref{randicexpression}, we have the desired result.
\end{proof}

Observation~\ref{septool} is used in Theorem~\ref{thm:upper}.

\begin{obs}\label{septool}
For $1 \le x < y < z$, we have
$$\left(\frac 1{\sqrt{x}} - \frac 1{\sqrt{z}} \right)^2 > \left(\frac 1{\sqrt{x}} - \frac 1{\sqrt{y}} \right)^2 + \left(\frac 1{\sqrt{y}} - \frac 1{\sqrt{z}} \right)^2.$$
\end{obs}
\begin{proof}
$$\left(\frac 1{\sqrt{x}} - \frac 1{\sqrt{z}} \right)^2 - \left(\frac 1{\sqrt{x}} - \frac 1{\sqrt{y}} \right)^2 - \left(\frac 1{\sqrt{y}} - \frac 1{\sqrt{z}} \right)^2= 2\left(\frac 1{\sqrt{x}} -\frac 1{\sqrt{y}} \right)
\left(\frac 1{\sqrt{y}} -\frac 1{\sqrt{z}} \right)>0.$$
\end{proof}

Now, we give a sharp upper bound for $R(G)$ in an $n$-vertex connected graph $G$ with given minimum and maximum degree.  Note that for a regular graph $G$, $R(G)=\frac{|V(G)|}2$. Thus we assume that $d< D$ in Theorem~\ref{thm:upper}. 
\begin{theorem}\label{thm:upper}
If $G$ is an $n$-vertex connected graph with minimum degree $d$ and maximum degree $D$, then $$R(G) \le \frac n2- \sum_{i=d}^{D-1}\frac 12 \left( \frac 1{\sqrt{i}} - \frac 1{\sqrt{i+1}}\right)^2.$$
Equality holds only for $G \in {\cal F}$.
\end{theorem}
\begin{proof} 
Let $V_d$ and $V_D$ be the sets of vertices with degree $d$ and $D$, respectively.
Among paths whose one end-vertex is in $V_d$ and the other is in $V_D$, consider a shortest path $P=x_0 ... x_l$, where $x_0 \in V_d$ and $x_l \in V_D$.  For $i \in [0,l-1]$, if $|d(x_i)-d(x_{i+1})| \ge 2$ (say $d(x_i) < d(x_{i+1})$), then by Observation~\ref{septool},
$$\left(\frac 1{\sqrt{d(x_i)}} - \frac 1{\sqrt{d(x_{i+1})}} \right)^2 > \left(\frac 1{\sqrt{d(x_i)}} - \frac 1{\sqrt{d(x_i)+1}} \right)^2 + \left(\frac 1{\sqrt{d(x_i)+1}} - \frac 1{\sqrt{d(x_{i+1})}} \right)^2$$
$$> \sum_{j=d(x_i)}^{d(x_{i+1})-1} \left( \frac 1{\sqrt{j}} - \frac 1{\sqrt{j+1}}\right)^2.$$

Note that for any positive integer $k$ between $d$ and  $D$,
there exists $i \in [0,l-1]$ such that $k \in [d(x_i), d(x_{i+1})]$,
since $P$ has end-vertices with degree $d$ and $D$ and is clearly connected.
Thus, by Theorem~\ref{randicexpression}, we have
\begin{align*}R(G) = \frac n2 - \sum_{uv \in E(G)} \frac 12\left(\frac 1{\sqrt{d(u)}} -\frac 1{\sqrt{d(v)}} \right)^2 &\le \frac n2 - \sum_{uv \in E(P)} \frac 12\left(\frac 1{\sqrt{d(u)}} -\frac 1{\sqrt{d(v)}} \right)^2\\[2mm]
&\le \frac n2- \sum_{i=d}^{D-1}\frac 12 \left( \frac 1{\sqrt{i}} - \frac 1{\sqrt{i+1}}\right)^2.
\end{align*}

Equality holds in this bound if and only if edges $uv$ with $d(u)\neq d(v)$ are only on the path $P$ and $d(x_{i+1})-d(x_{i})=0$ or 1. Note that $d(x_0)=d, d(x_1)=d+1, \ldots, d(x_{l-1})=D-1, d(x_{l})=D$. Thus $G$ must be in ${\cal F}$.
\end{proof}

\section*{Acknowledgements}
Suil O would like to thank the Chern Institute of Mathematics, Nankai University, for their generous hospitality. He was able to carry out part of this research during his visit there. Yongtang Shi is partially supported by the Natural Science
Foundation of Tianjin (No. 17JCQNJC00300) and the National Natural Science Foundation of China. The authors would like to thank Shenwei Huang for his discussion.

\end{document}